\documentclass[12 pt,a4paper,reqno]{amsart}
\usepackage{amssymb,latexsym,amsmath,amscd,amsthm,amsfonts, enumerate}
\usepackage{multirow}
\usepackage{color}
\usepackage[all]{xy}
\usepackage{etex}
\usepackage{caption}
\usepackage{soul}
\usepackage{float}
\usepackage{titletoc}
\usepackage[normalem]{ulem}
\usepackage{graphicx}
\usepackage[usenames,dvipsnames]{xcolor}
\usepackage{url}

\usepackage{tikz,tikz-cd}
\usepackage{mathrsfs}
\usepackage[all]{xy}
\usepackage{tikz}
\usepackage{extarrows}
\usepackage{tikz-cd}
\usetikzlibrary{calc}
\usetikzlibrary{matrix,arrows,decorations.pathmorphing}
\usetikzlibrary{snakes}
\usetikzlibrary{shapes.geometric,positioning}
\usetikzlibrary{arrows,decorations.pathmorphing,decorations.pathreplacing}
\usetikzlibrary{positioning,shapes,shadows,arrows,snakes}

\usepackage[colorlinks=true,pagebackref,hyperindex]{hyperref}
\newcommand\myshade{85}
\colorlet{mylinkcolor}{violet}
\colorlet{mycitecolor}{red}
\colorlet{myurlcolor}{cyan}

\hypersetup{
  linkcolor  = mylinkcolor!\myshade!black,
  citecolor  = mycitecolor!\myshade!black,
  urlcolor   = myurlcolor!\myshade!black,
  colorlinks = true,
}

\numberwithin{equation}{section}

\newtheorem{theorem}{Theorem}[section]
\newtheorem{proposition}[theorem]{Proposition}
\newtheorem{proposition-definition}[theorem]{Proposition-Definition}

\newtheorem{corollary}[theorem]{Corollary}
\newtheorem{lemma}[theorem]{Lemma}

\theoremstyle{definition}
\newtheorem{remark}[theorem]{Remark}
\newtheorem{example}[theorem]{Example}
\newtheorem{definition}[theorem]{Definition}

\definecolor{dark-green}{RGB}{14,150,2}
\definecolor{red}{RGB}{250,0,0}

\begin{document}

\title[Noncommutative Dominated Convergence Theorem]{Measure Convergence under Weak Domination for Unitarily Invariant Norms on Finite von Neumann Algebras  }

\author{Yanni Chen}
\address{School of Mathematics and Statistics, Shaanxi Normal University,
Xi'an 710119, China }
\email{yanni.chen@snnu.edu.cn}
\author{Don Hadwin}
\address{Department of Mathematics, University of New Hampshire, Durham, NH 03824, U.S.A.}
\email{don@unh.edu}
\author{Junhao Shen}
\address{Department of Mathematics, University of New Hampshire, Durham, NH 03824, U.S.A.}
\email{junhao.shen@unh.edu}
\author{Ye Zhang}
\address{School of Mathematics and Statistics, Shaanxi Normal University,
Xi'an 710119, China }
\email{zhangye@snnu.edu.cn}

{


\date{\today}


\begin{abstract}
The classical Lebesgue dominated convergence theorem (DCT) is a cornerstone in commutative measure theory, while its extension to the noncommutative setting has remained a longstanding challenge due to the barriers posed by the  noncommutative structure.  Let \(\alpha\) be a unitarily invariant norm defined on a diffuse tracial von Neumann algebra \((\mathcal{M}, \tau)\), let $L^\alpha(\mathcal M, \tau)$ be  the completion of $\mathcal M$ with respect to the norm $\alpha$(which includes all $L^p(\mathcal M, \tau)$ as special cases). In this paper,  by replacing ``almost everywhere convergence"  with Nelson's convergence in measure, we establish a noncommutative analog of the dominated convergence theorem for $L^\alpha(\mathcal M, \tau)$  under a { significantly weaker dominating condition} than those previously considered. For instance, in a \(II_1\) factor \(\mathcal{M}\), if \(P\) is a projection with \(\tau(P)=\frac{1}{n}\), then \(n^2 P\) ``dominates" every operator of norm \(\leq n\).   As
an application, we obtain a complete characterization of center-preserving $*$-automorphism of finite von Neumann algebras, which unifies and extends the earlier works of Kadison and Haagerup on the structure of automorphisms.

\end{abstract}

\maketitle
\setcounter{tocdepth}{1}

\tableofcontents

\section{Introduction}\label{Introductions}

Driven by the specific analytic situations encountered in group theory, in operator theory and in quantum field theory,  Irving Segal\cite{Segal} introduced the notion of noncommutative integration  on a von Neumann algebra in 1953. His foundational work was subsequently advanced by  Kunze\cite{Kunze}and Stinepring\cite{Stinespring}.  Haagerup\cite{Haagerup2} in his landmark paper established the theory of noncommutative $L^p$-spaces, in which the classical $L^p$-spaces are replaced by spaces of measurable operators affiliated with a von Neumann algebra $\mathcal M$ equipped with a faithful normal tracial state $\tau.$ Since then, the theory has been extensively developed and applied by Fack, Kosaki, Pisier, Junge, Xu, Sukochev and many others(see \cite{BM}, \cite{BCO},  \cite{Dodds}, \cite{HLW}, \cite{Junge}, \cite{Kosaki}, \cite{MW}, \cite{SSZ}, \cite{Xu}). The recent survey by Pisier and Xu\cite{PX} presents a  complete picture on noncommutative integration and contains a lot of references. This theory is still a very active subject of investigation.

The Lebesgue Dominated Convergence Theorem (DCT)\cite{Rudin} is one of the most fundamental and frequently used theorem in commutative measure theory. It asserts that if a sequence of measurable functions $\{f_n\}$ converges pointwise almost everywhere to $f,$ and if there exists an integrable function g such that $ |f_n| \leq g $ for all $n\in \mathbb N$, then
$$\lim_{n\rightarrow \infty}\int f d\mu=\int\lim_{n\rightarrow \infty}f_n \ d\mu.$$
The following picture can well describe this principle:

\vspace{1mm}
{
 \ \ \quad\quad\quad \includegraphics[height=4.5cm,    width=12.5cm]{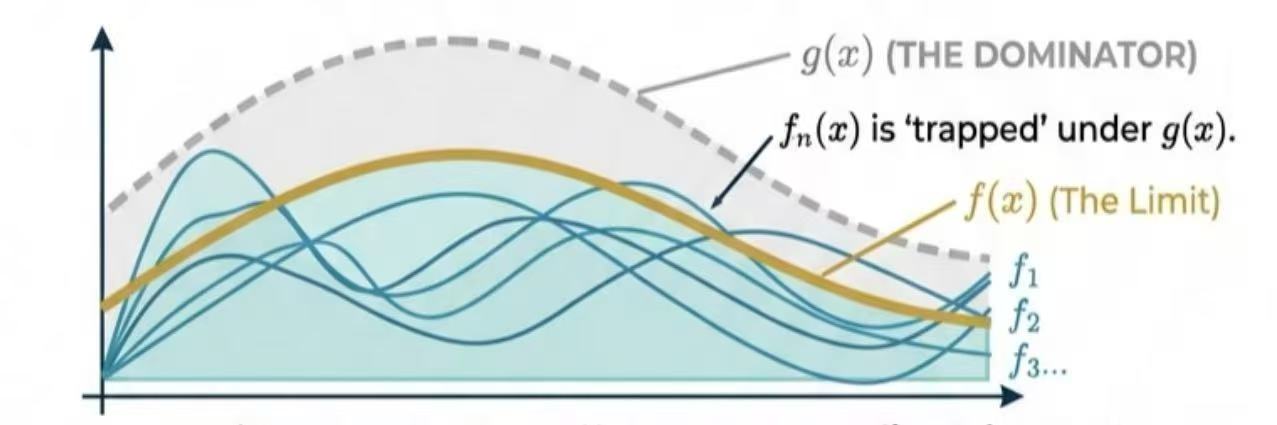}
 \vspace{1mm}}

Even as the sequence of functions varies randomly, an invisible ``ceiling" function $g$ prevents the total area from ``escaping" or collapsing. Owing to its powerful and importance, this theorem  has become an indispensable tool in probability theory, harmonic analysis, partial differential equations, and virtually every branch of applied mathematics where limits and integrals interact(see \cite{CoenZoli}, \cite{GhoshMondal}, \cite{MHT}).

In this context, transferring  classical measure-theoretic results to the noncommutative setting has become a central research topic. A natural and long-standing problem therefor arises: {Can the classical DCT be extended to the noncommutative setting? If so, what is the appropriate notion of  convergence mode, and what form does the  dominating operator take?}

Unitarily invariant norms, which include the classical $L^p$-norms, Ky Fan norms, and Orlicz norms were introduced by von Neumann\cite{vNeumann} for the purpose of metrizing matrix spaces. These norms are widely utilized in operator theory and quantum information theory, owing to their invariance under unitary conjugations ($\|U^*AU\| =\| A\|$ for any unitary  $U$),  reflecting the rotational symmetry of operators independent of their representation. Later, Fack and Kosaki\cite{FK} generalized the notion of singular values ($s$-numbers) to measurable operators, providing a powerful tool for studying unitarily invariant norms on noncommutative $L^p$-spaces. In 2008, Fang, Hadwin, Nordgren, and Shen\cite{FHNS} systematically studied tracial gauge norms on finite von Neumann algebras, establishing connections between unitarily invariant norms and the weak Dixmier property.

Motivated by the relation between von Neumann's theorems and the analogies between noncommutative and classical $L^p$-spaces, in this paper, we consider noncommutative $L^p$-spaces associated with unitarily invariant norms on a finite von Neumann algebra $\mathcal M$ and prove a noncommutative version of Lebesgue dominated convergence theorem.

 However, extending  the classical DCT to this noncommutative setting encounters several fundamental challenges:
\begin{enumerate}
 \item Absence of pointwise evaluation. For unbounded measurable operators, there is no direct analogue of ``almost everywhere pointwise convergence." The classical notion must be replaced by a suitable operator-topological convergence.

\item Noncommutativity. The commutative structure of functions is lost; operators do not commute, and inequalities between operators must be interpreted in the sense of positive elements.

 \item Dominating conditions. Finding an appropriate noncommutative analogue of the dominating integrable function is non-trivial, as there is no pointwise order structure.

 \item Lack of lattice properties for unitarily invariant norms. Unlike the classical $L^p$-norms, the class of unitarily invariant norms does not enjoy a lattice monotonicity. The classical proofs for the $L^p$-case rely on the $L^2$-result and take cases when $p\leq 2$ and $2<p$ have no analogue, hence the tools and techniques developed for $L^p$-spaces are no longer available.
\end{enumerate}

For convenience of description, let us recall some basic notations. A finite von Neumann algebra $(\mathcal M, \tau)$ is a von Neumann algebra $\mathcal M$ with a faithful normal tracial state $\tau.$ A unitarily invariant norm $\alpha$  on a finite von Neumann algebra  $\mathcal M$ is a norm satisfying $\alpha(UTV)=\alpha(T)$ for all $T\in \mathcal M$ and  all unitaries $U, V\in \mathcal M.$  We let $L^\alpha(\mathcal M, \tau)$ be  the completion of $\mathcal M$ with respect to the norm $\alpha.$ In 1974, E. Nelson\cite{Nelson} defined the topology of convergence in measure on   $\mathcal M.$  This serves as a natural noncommutative analog of classical convergence in measure, which is strictly weaker than pointwise almost-everywhere convergence.  We  let $L^0(\mathcal M, \tau)$ be the completion of  $\mathcal M$ with respect to this topology, which consists of all $\tau$-measurable operators. For any $0<p<\infty,$ noncommutative $L^p(\mathcal M, \tau)$ can be injected into $L^0(\mathcal M, \tau).$

 The primary objective of this work is to overcome the above technical barriers and establish a universal framework that  covers all situations with significantly weaker hypotheses--both on the convergence mode and on the  dominating condition. In the noncommutative setting, several analogs of  ``almost everywhere convergence" have been proposed: Lance\cite{Lance} invented the uniform convergence, and Yeadon\cite{Yeadon} introduced the bilateral almost uniform convergence. In this paper, we first employ Nelson's sense of convergence in measure on finite von Neumann algebras as the replacement of the classical pointwise almost everywhere convergence,  and we establish a fully generalized dominated convergence theorem for $L^\alpha(\mathcal M, \tau).$
 Second, we introduce a relaxed notion of the  ``dominating operators" as $\alpha$-dominating set, which is summarized as below(see Definition \ref{equicontinuous} and Definition \ref{alpha dominating set}):

\begin{definition}\label{equicontinuous}
Suppose \(\alpha\) is a continuous normalized unitarily invariant norm on \(\mathcal{M}\). Then
\begin{enumerate}
\item\(\mathcal{S} \subset L^{\alpha}(\mathcal{M}, \tau)\) is called \(\alpha\)-$equicontinuous,$ if for every \(\varepsilon > 0\), there is a \(\delta > 0\) such that, for every \(S \in \mathcal{S}\) and every projection \(P \in \mathcal{M}\),
\[
\tau(P) < \delta \Longrightarrow \alpha(PS) \leq \varepsilon.
\]

\item  The set of $\alpha$-dominating functions \(d_{\mathcal{S}}: (0, 1] \rightarrow (0, 1]\)  is defined by
\[
d_{\mathcal{S}}(\varepsilon) = \sup \left\{ 0 < \delta \leq 1 : \forall S \in \mathcal{S},\  \forall P \in \mathcal{P}(\mathcal{M}), \ [\tau(P) < \delta \Longrightarrow \alpha(PS) \leq \varepsilon] \right\}.
\]
\item An element \(T \in \mathcal{L}^{0}(\mathcal{M}, \tau)\) is \(\alpha\)-dominated by \(\mathcal{S}\) if and only if
\[
\forall P \in \mathcal{P}(\mathcal{M}), \quad [\tau(P) < d_{\mathcal{S}}(\varepsilon) \Longrightarrow \alpha(PT) \leq \varepsilon],
\]
\item The collection of ``dominating operators", denoted by $\mathcal{D}_{\alpha}(\mathcal{S}),$ is defined  to be
 $$\mathcal{D}_{\alpha}(\mathcal{S}) = \mathcal{D}_{\alpha}(\mathcal{S}, \mathcal{M}) = \left\{ T \in \mathcal{L}^{0}(\mathcal{M}, \tau) : T \text{ is \(\alpha\)-dominated by \(\mathcal{S}\)} \right\}.$$
\end{enumerate}
\end{definition}
Such $\mathcal{D}_{\alpha}(\mathcal{S})$  serves as a powerful notion of domination that is substantially weaker than the classical pointwise domination.

Based on the equivalent characterizations of Nelson's convergence in measure topology (see Lemma \ref{Convergence in measure Lemma}), we construct a well-defined Borel functional calculus for normal operators in $L^0(\mathcal M, \tau)$ (See Theorem \ref{Borel functional calculus}). Then together with the connection between the $\alpha$-convergence and  convergence in measure on $\mathcal{D}_{\alpha}(\mathcal{S})$ (See Theorem \ref{alpha convergence implies in measure}), we obtain the following main result (see Theorem \ref{DCT}).

\begin{theorem}\label{DCT}
Suppose \(\mathcal{M}\) is a diffuse von Neumann algebra with a faithful normal trace \(\tau\), \(\alpha\) is a continuous normalized unitarily invariant norm on \(\mathcal{M}\), and \(\mathcal{S} \subset L^{\alpha}(\mathcal{M}, \tau)\) is \(\alpha\)-equicontinuous. Then

\begin{enumerate}
    \item \(\mathcal{S} \subset \mathcal{D}_{\alpha}(\mathcal{S})\).
    \item \(\mathcal{D}_{\alpha}(\mathcal{S})\) is convex.
\end{enumerate}

\begin{enumerate}
    \item[(3)] If \(T \in \mathcal{D}_{\alpha}(\mathcal{S})\) and \(A, B \in \mathcal{M}\) with \(\|A\|, \|B\| \leq 1\), then \(ATB \in \mathcal{D}_{\alpha}(\mathcal{S})\).

    \item[(4)] If \(S, S_{1}, S_{2}, \ldots \in \mathcal{D}_{\alpha}(\mathcal{S})\), then \(S_{n} \to S\) in measure if and only if
    \[
    \alpha(S_{n} - S) \to 0.
    \]

    \item[(5)] The \(\alpha\)-topology and the topology of convergence in measure coincide on \(\mathcal{D}_{\alpha}(\mathcal{S})\).

    \item[(6)] \(\mathcal{D}_{\alpha}(\mathcal{S}) \subset L^{\alpha}(\mathcal{M},\tau)\).

    \item[(7)] \(\mathcal{D}_{\alpha}(\mathcal{S})\) is closed in convergence in measure and in the \(\alpha\)-topology.

    \item[(8)] If \(\gamma\) is a \(\ast\)-automorphism of \(\mathcal{M}\) such that \(\tau \circ \gamma = \tau\) and \(\alpha \circ \gamma = \alpha\), then \(\gamma(T) \in \mathcal{D}_{\alpha}(\mathcal{S})\) whenever \(T \in \mathcal{D}_{\alpha}(\mathcal{S})\).
\end{enumerate}

\end{theorem}

Theorem 1.1 not only unify and extend the classical dominated convergence theorem to the noncommutative analysis, but also provides new techniques and new tools for noncommutative analysis and operator spaces, such as  constructing maximal functions or maximal  inequalities for noncommutative martingale theory.

As an application of our noncommutative DCT, we provide a complete equivalent characterization of $*$-automorphisms on finite von Neumann algebras (see Theorem 5.7).

\begin{theorem}
Suppose \(\mathcal{M}\) is a finite von Neumann algebra and \(\gamma\) is a \(*\)-automorphism of \(\mathcal{M}\). The following are equivalent:

\begin{enumerate}
    \item[(1)] \(\gamma(S) = S\) for every \(S \in \mathcal{Z}(\mathcal{M})\), i.e., \(\gamma\) fixes \(\mathcal{Z}(\mathcal{M})\).

    \item[(2)] \(\Phi \circ \gamma = \Phi\), where \(\Phi: \mathcal{M} \rightarrow \mathcal{Z}(\mathcal{M})\) is the center-valued trace on \(\mathcal{M}\).

    \item[(3)] For every \(A = A^{*} \in \mathcal{M}\), there is a sequence \(\{W_{n}\}\) of unitary operators in \(\mathcal{M}\) such that \(\|W_{n}^{*}AW_{n} - \gamma(A)\| \to 0\).

    \item[(4)] \(\alpha(\gamma(T)) = \alpha(T)\) for every unitarily invariant norm \(\alpha\) on \(\mathcal{M}\).

    \item[(5)] \(\tau \circ \gamma = \tau\) for every normal tracial state \(\tau\) on \(\mathcal{M}\).
\end{enumerate}
\end{theorem}

Among the five equivalent conditions in Theorem 1.3, the equivalence $(1)\Leftrightarrow (2)$ is a classical consequence of Kadison\cite{Kadison} on the center-valued trace. The equivalence $(4)\Leftrightarrow(5)$ is related to the work of Haagerup\cite{Haagerup1} on the standard form of von Neumann algebras, and the equivalence $(1)\Leftrightarrow(5)$ is well-known. Our contribution in this paper is to provide a new norm-theoretic criterion:
the class of center-preserving $*$-automorphisms (equivalently, $*$-automorphisms leaves trace invariant) coincides with those $*$-automorphisms preserving  every unitarily invariant norm on $\mathcal M$.

 The organization of the paper is structured as follows.  In Section 2, we give a  characterization of Nelson's convergence in measure, which yields to a new proof of  a result of Haagerup on a tracial embeddings  and the extension to a unital $*$-isomorphism \(\hat{\pi}\) on the algebra \(\mathcal{L}^{0}[0,1]\). This easily leads to a natural exposition of the Borel functional calculus for normal elements in \(\mathcal{L}^{0}(\mathcal{M}, \tau)\) in which convergence in measure is easily handled. Section 3  discusses unitarily invariant norms and tracial gauge norms, and establishes an equivalent characterization of equimeasurability for positive operators on a $II_1$ factor. In Section 4, we establish a fundamental relation between  convergence with respect to unitarily invariant norms and convergence in measure. This result is essential for embedding $L^\alpha(\mathcal M, \tau)$ into $L^0(\mathcal M, \tau).$ In Section 5, based on the notions of $\alpha$-equicontinuity and $\alpha$-domination, we  presents the noncommutative dominated convergence theorem for unitarily invariant norms, where both the notions of convergence and the dominating condition are greatly weakened. Moreover, we also characterize the $*$-automorphisms on any finite von Neumann algebra $\mathcal{M}$ that are isometries with respect to every unitarily invariant norm or, equivalently, preserve every tracial state, as those that fix the center.

This paper is almost self-contained and we do not assume any backgrounds on noncommutative $L^p$-theory.

\subsection*{Acknowledgments}

This paper is supported in part by the National Natural Science Foundation of China (Grant No. 12271321)
and by a grant from Simons Foundation. This work was initiated while Y. Chen was in the Department of Mathematics, University of New Hampshire during 2008 to 2015.

\section{Tracial embeddings, Convergence in Measure, and  Borel Functional Calculus}
\label{sec:3sets}

Throughout this section, we assume given a von Neumann algebra $\mathcal M$ on a Hilbert space $\mathcal H$, with a faithful normal tracial state $\tau.$

\subsection{Tracial embeddings }

 If \(\mathcal{C}\) is a maximal chain of projections in \(\mathcal{M}\), then we can write
\[
\mathcal{C}=\left\{P_{t}: 0 \leq t \leq 1\right\},
\]
with \(\tau(P_{t}) = t\) for every \(0 \leq t \leq 1\). Suppose \(\mu\) is Lebesgue measure on \([0, 1]\). It is
well-known that the map
\[
\chi_{[0, t)} \mapsto P_{t}
\]
 extends uniquely to a unital \(*\)-homomorphism \(\pi: {L}^{\infty}[0,1] \to {W}^{*}(\mathcal{C})\) such that, for every \(f \in {L}^{\infty}[0,1]\),
\[
\tau(\pi(f))=\int_{0}^{1} f \mathrm{d}\mu.
\]
We call such a map a $tracial\  embedding$ of \({L}^{\infty}[0,1]\) into \(\mathcal{M}\). If \(\mathcal{N}\) is a diffuse countably generated commutative von Neumann subalgebra of \(\mathcal{M}\), then there is a self-adjoint element \(A \in \mathcal{N}\) such that \(\mathcal{N}={W}^{*}(A)\). Then
\[
\mathcal{C}=\left\{\chi_{[0, s)}(A): 0 \leq s \leq 1\right\}
\]
is a maximal chain of projections in \(\mathcal{N}\) and \({W}^{*}(A)={W}^{*}(\mathcal{C})=\mathcal{N}\). Thus there is a tracial embedding of \({L}^{\infty}[0,1]\) onto \(\mathcal{N}\). If \(\mathcal{M}\) is separable (i.e., acts on a separable Hilbert space), then every diffuse commutative von Neumann subalgebra of \(\mathcal{M}\) is countably generated and is therefore the image of a tracial embedding of \({L}^{\infty}[0,1]\).

\subsection{Convergence in measure}

In \cite{Nelson}, E. Nelson defined the topology of convergence \ in \ measure on \(\mathcal{M}\), which is a central concept used in our work. Roughly speaking, it is appropriate  to define noncommutative ``measurable operators" in $\mathcal M$.

\begin{definition}
For $\varepsilon>0,$ let
$$N(\varepsilon)=\{T\in \mathcal M:  \text{for some projection  } P\in \mathcal M, \ \tau(P)\leq \varepsilon,\  \text{and} \ \|T(1-P)\|\leq \varepsilon \}$$
be a fundamental system of neighborhoods of $0$ in $\mathcal M.$  It is called the measure topology of $\mathcal M.$

\end{definition}

\begin{remark}
Since there is a countable neighborhood base at the origin, the topology on \(\mathcal{M}\) is metrizable. We define the notion of convergence in measure ($\tau$-convergence)  as follows:
 $$T_{\lambda} \to T \  \mbox{in measure} \ (i.e.,
T_{\lambda} \stackrel{\tau}{\to} T)\Longleftrightarrow \forall
 \varepsilon>0,,  \exists  \lambda_0 \ s.t.\  \forall \lambda\geq \lambda_0,\  T_\lambda\in N(\varepsilon).$$

\end{remark}

Let \(\mathcal{L}^{0}(\mathcal{M}, \tau)\) denote the completion of \(\mathcal{M}\) with respect to convergence in measure. If \(S \subset \mathcal{L}^{0}(\mathcal{M}, \tau)\), we denote the closure of \(S\) in measure by \(S^{-\tau}\).

\begin{proposition} Suppose  $T\in \mathcal M.$ Then
\begin{enumerate}\item $|T|\in  N({\varepsilon})\Longleftrightarrow T\in N(\varepsilon)\Longleftrightarrow T^*\in N(\varepsilon ) ;$
\item \(|\tau(ATB)| \leq \|A\|\|T\|_{1}\|B\|\) for all $A, B\in \mathcal M.$
 \end{enumerate}

\end{proposition}

\begin{proof}

(1) Since \(\mathcal{M}\) is a finite von Neumann algebra, every \(T\) in \(\mathcal{M}\) has a polar form \(T = U(T^{*}T)^{1/2}\) with \(U\) unitary. We define \(|T| = (T^{*}T)^{1/2}\), so the polar decomposition of \(T\) is \(T = U|T|\). Note that \[\|T(1 - P)\| = \||T|(1 - P)\|\] for every projection \(P\). Hence \(T \in N(\varepsilon)\) if and only if \(|T| \in  N(\varepsilon)\).

On the other hand, since \[|T^{*}| = (TT^{*})^{1/2} = U|T|U^{*},\] it follows that \(T \in  N(\varepsilon)\) if and only if \(T^{*} \in N(\varepsilon)\).

(2) Recall that for \(T \in \mathcal{M}\), \(\|T\|_{1} = \tau(|T|)\) is a norm on \(\mathcal{M}\), and for every \(A, B \in \mathcal{M}\), we have \(|\tau(ATB)| \leq \|ATB\|_{1} \leq \|A\|\|T\|_{1}\|B\|\).

\end{proof}

It is well-known that a sequence \(\{f_{n}\}\) of measurable functions on \([0, 1]\) converges in measure to 0 if and only if
\[
\lim_{n \to \infty} \int_{[0,1]} \frac{|f_{n}|}{1 + |f_{n}|} \, d\mu = 0,
\]
and that the space \(\mathcal{L}^{0}[0, 1]\) of all measurable functions is a metric space with the metric
\[
d(f, g) = \int_{[0,1]} \frac{|f - g|}{1 + |f - g|} \, d\mu.
\]

The analogue in \((\mathcal{M}, \tau)\) is
\[
d(S, T) = \tau\left(|S - T|[1 + |S - T|]^{-1}\right),
\]
although this might not be a metric when \(\mathcal{M}\) is not commutative. However, the analogy still works.

\begin{lemma}\label{Convergence in measure Lemma}
Suppose \(\{T_{\lambda}\}\) is a net in \(\mathcal{M}\). The following are equivalent:
\begin{enumerate}
\item \(T_{\lambda} \stackrel{\tau}{\to} 0\).
\item For every \(\varepsilon > 0\), \(\tau\left(\chi_{[\varepsilon, \infty)}(|T_{\lambda}|)\right) \to 0\).
\item There is a net \(\{P_{\lambda}\}\) of projections in $\mathcal M $ such that
\[
\tau(P_{\lambda}) \to 0 \quad \text{and} \quad \left\||T_{\lambda}|(1 - P_{\lambda})\right\| \to 0.
\]
\item \(\lim_{\lambda} \tau\left(|T_{\lambda}|[1 + |T_{\lambda}|]^{-1}\right) = 0\).
\end{enumerate}
\end{lemma}

\textbf{Proof}.
\((2) \Rightarrow (1)\). Assume (2) holds and \(\varepsilon > 0\). We have \(\tau\left(\chi_{[\varepsilon, \infty)}(|T_{\lambda}|)\right) \to 0\). Thus there is a $\lambda_0$ such that $\lambda\geq \lambda_0$ implies
$\tau(\chi_{[\varepsilon, \infty)}(|T_{\lambda}|)<\varepsilon.$ Also, for every $\lambda \geq \lambda_{0}$,
\[
\left\||T_{\lambda}|(1 - \chi_{[\varepsilon, \infty)}(|T_{\lambda}|))\right\| \leq \varepsilon,
\]
which implies \(T_{\lambda} \in N(\varepsilon)\). Hence \(T_{\lambda} \stackrel{\tau}{\to} 0\).

\((1) \Rightarrow (3)\). Suppose \(0 < \varepsilon, \delta\) and \(T \in N(\delta)\). Then \(|T| \in  N(\delta)\). Thus there is a projection \({P} \in \mathcal{M}\) such that \(\tau(\mathcal{P}) \leq \delta\) and \(\||T|(1 - {P})\| \leq \delta\). Since
 \[0 \leq \varepsilon^{2}\left[\chi_{[\varepsilon, \infty)}(|T|)\right]^{2} \leq \chi_{[\varepsilon, \infty)}(|T|)|T|^{2},\]we have
\[
0 \leq \varepsilon^{2}(1 -{P})\left[\chi_{[\varepsilon, \infty)}(|T|)\right]^{2}(1 - {P}) \leq (1 - {P})\left(\chi_{[\varepsilon, \infty)}(|T|)|T|^{2}\right)(1 - {P}).
\]
Taking norms, we have
\[
\varepsilon^{2}\left\| \left[\chi_{[\varepsilon, \infty)}(|T|)\right](1 - {P})\right\|^{2} \leq \left\| \left(\chi_{[\varepsilon, \infty)}(|T|)|T|\right)(1 - {P})\right\|^{2} \leq \delta^{2}.
\]
Thus,
\[
\left\| \chi_{[\varepsilon, \infty)}(|T|) - \chi_{[\varepsilon, \infty)}(|T|){P} \right\| = \left\| \left[\chi_{[\varepsilon, \infty)}(|T|)\right](1 - {P}) \right\| \leq \delta / \varepsilon.
\]
Taking adjoints, we have
\begin{align*}
\tau\left(\chi_{[\varepsilon, \infty)}(|T|)\right) &\leq \tau\left(\chi_{[\varepsilon, \infty)}(|T|) - \chi_{[\varepsilon, \infty)}(|T|){P}\right) + \tau\left(\chi_{[\varepsilon, \infty)}(|T|){P}\right) \\
&\leq \left\| \chi_{[\varepsilon, \infty)}(|T|) - \chi_{[\varepsilon, \infty)}(|T|){P} \right\| + \left\| \chi_{[\varepsilon, \infty)}(|T|){P} \right\|\\
& \leq \delta / \varepsilon + \left\| \chi_{[\varepsilon, \infty)} \right\| \| {P} \|_{1} \\
&\leq \delta / \varepsilon + \tau({P}) < \delta + \delta = 2\delta.
\end{align*}
Hence, if (1) holds, then (3) must be true.

\((3) \Rightarrow (4)\). Suppose (3) holds. Then we have
\begin{align*}
\tau\left(\left[1 + \left|T_{\lambda}\right|\right]^{-1}\left|T_{\lambda}\right|\right)
&= \left|\tau\left(\left(\left[1 + \left|T_{\lambda}\right|\right]^{-1}\right)\left|T_{\lambda}\right|\mathcal{P}_{\lambda}\right| + \left|\tau\left((\left[1 + \left|T_{\lambda}\right|\right]^{-1}\right)\left|T_{\lambda}\right|\left(1 - \mathcal{P}_{\lambda}\right)\right)\right| \\
&\leq \left\| \left(\left[1 + \left|T_{\lambda}\right|\right]^{-1}\right)\left|T_{\lambda}\right|\right\| \|\mathcal{P}_{\lambda}\|_{1} + \left\| \left(\left[1 + \left|T_{\lambda}\right|\right]^{-1}\right)\right\|_{1} \left\| \left|T_{\lambda}\right|\left(1 - \mathcal{P}_{\lambda}\right)\right\| \\
&\leq \tau\left(\mathcal{P}_{\lambda}\right) + \left\| \left|T_{\lambda}\right|\left(1 - \mathcal{P}_{\lambda}\right)\right\| \to 0.
\end{align*}
Since $\tau$ is tracial, it follows that \(\lim_{\lambda} \tau\left(|T_{\lambda}|[1 + |T_{\lambda}|]^{-1}\right) = 0\).

\((4) \Rightarrow (2)\). Suppose \(\tau\left(\left|T_{\lambda}\right|\left[1 + \left|T_{\lambda}\right|\right]^{-1}\right) \to 0\). If \(t \geq \varepsilon > 0\), we have \[\frac{t}{1 + t}=1 - \frac{1}{1 + t} \geq \frac{\varepsilon}{1 + \varepsilon}.\] It follows that
\[
\left|T_{\lambda}\right|\left[1 + \left|T_{\lambda}\right|\right]^{-1} \geq \left|T_{\lambda}\right|\left[1 + \left|T_{\lambda}\right|\right]^{-1} \chi_{[\varepsilon, \infty)}\left(\left|T_{\lambda}\right|\right) \geq \frac{\varepsilon}{1 + \varepsilon} \chi_{[\varepsilon, \infty)}\left(\left|T_{\lambda}\right|\right).
\]
Thus
\[
\tau\left(\chi_{[\varepsilon, \infty)}\left(\left|T_{\lambda}\right|\right)\right) \leq \frac{1 + \varepsilon}{\varepsilon} \tau\left(\left|T_{\lambda}\right|\left[1 + \left|T_{\lambda}\right|\right]^{-1}\right) \to 0.
\]
Hence the statement (2) holds.

\subsection{Borel functional calculus}
E. Nelson\cite{Nelson} and U. Haagerup prove that \(\mathcal{L}^{0}(\mathcal{M}, \tau)\) is a \(\ast\)-algebra with the operations of multiplication, scalar multiplication, addition, and adjoint are all continuous. Furthermore, the elements of \(\mathcal{L}^{0}(\mathcal{M}, \tau)\) can be viewed as closed densely defined linear operators on \({L}^{2}(\mathcal{M}, \tau)\) such that if \(T \in \mathcal{L}^{0}(\mathcal{M}, \tau)\) is normal, then all spectral projections of \(T\) are in \(\mathcal{M}\). Moreover, every \(T \in \mathcal{L}^{0}(\mathcal{M}, \tau)\) has a polar decomposition
\[
T = U\left(T^{*}T\right)^{1/2},
\]
where \(U\) is a unitary operator in \(\mathcal{M}\). If \(0 \leq T \in \mathcal{L}^{0}(\mathcal{M}, \tau)\), then the chain of spectral projections of \(T\) corresponding to the intervals \([0, s)\) for \(s \in [0, \infty)\) is contained in a maximal chain \(\mathcal{C}\) of projections in \(\mathcal{M}\), and \(T \in {W}^{*}({\mathcal C})^{-\tau}\).

The following result, with a proof different from ours, was obtained by U. Haagerup\cite{Terp}. Recall that if \(X\) and \(Y\) are metric spaces and \(f: X \to Y\) is a bijection such that a sequence \(\{x_n\}\) is Cauchy in \(X\) if and only if \(\{f(x_{n})\}\) is Cauchy in \(Y\), then \(f\) extends uniquely to a homeomorphism \(\hat{f}\) from the completions of \(X\) and \(Y\) so that \(\hat{f}\) and \(\hat{f}^{-1}\) send Cauchy sequences to Cauchy sequences. The following result is, therefore, an immediate consequence of Lemma \ref{Convergence in measure Lemma}.

\begin{theorem}\label{Borel functional calculus}
Suppose \(\pi: {L}^{\infty}[0,1] \to \mathcal{M}\) is a unital \(*\)-homomorphism such that, for every \(f \in {L}^{\infty}[0,1]\),
\[
\tau(\pi(f)) = \int_{[0,1]} f \, \mathrm{d}\mu.
\]
Then \(\pi\) extends to a unital $*$-isomorphism \(\hat{\pi}: \mathcal{L}^{0}[0,1] \to \pi({L}^{\infty}[0,1])^{-\tau}\) that is a homeomorphism with respect to convergence in measure, such that, if \(f \geq 0\) is in \(\mathcal{L}^{0}[0,1]\), then
\[
\tau(\hat{\pi}(f)) = \int_{[0,1]} f \, \mathrm{d}\mu,
\]
where $\mathcal{L}^{0}[0,1]$ denotes all measurable functions on $[0,1].$
\end{theorem}

The above theorem gives rise to a functional calculus for normal elements of \(\mathcal{L}^{0}(\mathcal{M}, \tau)\).

Suppose \(T \in \mathcal{L}^{0}(\mathcal{M}, \tau)\) and \(T \in {W}^{*}(\mathcal{C})^{-\tau}\) for some maximal chain \(\mathcal{C}\) of projections in \(\mathcal{M}\). Then there is a tracial embedding \(\pi: {L}^{\infty}[0,1] \to {W}^{*}(\mathcal{C})\) that uniquely extends to an isomorphism \(\hat{\pi}: \mathcal{L}^{0}([0,1]) \to {W}^{*}(\mathcal{C})^{-\tau}\). Suppose \(f \in \mathcal{L}^{0}([0,1])\) with \(\hat{\pi}(f) = T\), and \(\varphi: \mathbb{C} \to \mathbb{C}\) is a Borel measurable function. We see that \(\varphi \circ f \in \mathcal{L}^{0}([0,1])\). We want to define \[\varphi(T) = \hat{\pi}(\varphi \circ f).\] To do this we need to show that the definition is independent of the choice of tracial embedding \(\pi\).

\begin{lemma}
Suppose \(\pi, \rho: {L}^{\infty}[0,1] \to \mathcal{M}\) are tracial embeddings, \(f, g \in \mathcal{L}^{0}([0,1])\), and \(\hat{\pi}(f) = \hat{\rho}(g) = T\). Then for every measurable \(\varphi: \mathbb{C} \to \mathbb{C}\),
\[
\hat{\pi}(\varphi \circ f) = \hat{\rho}(\varphi \circ g).
\]
\end{lemma}

\textbf{Proof}. Let \(\lambda\) be the measure \(\mu f^{-1} + \mu g^{-1}\); i.e., for every Borel subset of \(\mathbb{C}\),
\[
\lambda(E) = \mu\left(f^{-1}(E)\right) + \mu\left(g^{-1}(E)\right).
\]
It follows from Lusin's theorem and the Stone-Weierstrass theorem that there is a sequence \(\{p_{n}(z, \bar{z})\}\) of polynomials such that, for every \(z \in \mathbb{C}\),
\[
\lim_{n \to \infty} p_{n}(z, \overline{z}) = \varphi(z) \quad \text{a.e.} \, (\lambda).
\]
Thus \(p_{n} \circ f \to \varphi \circ f\) and \(p_{n} \circ g \to \varphi \circ g\) a.e. \((\mu)\). Since almost everywhere convergence implies convergence in measure for finite measure spaces, it follows that
\[
p_{n}(T, T^{*}) = p_{n}\left(\hat{\pi}(f), \hat{\pi}(f)^{*}\right) = \hat{\pi}\left(p_{n} \circ f\right) \stackrel{\tau}{\to} \hat{\pi}(\varphi \circ f),
\]
and
\[
p_{n}(T, T^{*}) = p_{n}\left(\overline{\rho}(g), \overline{\rho}(g)^{*}\right) = \overline{\rho}\left(p_{n} \circ g\right) \stackrel{\tau}{\to} \overline{\rho}(\varphi \circ g).
\]

This means that our definition of \(\varphi(T)\) when \(T\) is a normal element of \(\mathcal{L}^{0}(\mathcal{M}, \tau)\) makes sense. The following corollary states a relation between the functional calculus and the topology of convergence in measure.

\begin{corollary}
If \(T\) is a normal element of \(\mathcal{L}^{0}(\mathcal{M}, \tau)\) and \(\{\varphi_{n}\}\) is a sequence of measurable functions on \(\mathbb{C}\) such that \(\varphi_{n} \to \varphi\) pointwise, then \(\varphi_{n}(T) \stackrel{\tau}{\to} \varphi(T)\).
\end{corollary}

Suppose \(f: [0,1] \to [0, \infty)\) is measurable. Then there exists a unique nonincreasing right-continuous function \(f^{*}: [0,1] \to [0, \infty)\) such that \(f^{*} = f \circ \gamma\) for some invertible measure-preserving transformation \(\gamma: [0,1] \to [0,1]\). The function \(f^{*}\) is called the nonincreasing
rearrangement of  \(f\).

The formula for \(f^{*}\) is given by
\[
f^{*}(x) =
\begin{cases}
\sup \{ y : \mu(\{ f > y \}) > x \} & \text{if } 0 \leq x < 1, \\
\text{ess-inf}(f) & \text{if } x = 1.
\end{cases}
\]
We can also write \(f^{*}\) as
\[
f^{*}(x) =
\begin{cases}
\sup \left\{ y : \int_{0}^{1} \chi_{(y, \infty)} \circ f \, \mathrm{d}\mu > x \right\} & \text{if } 0 \leq x < 1, \\
\left\| \frac{1}{1 + f} \right\|_{\infty}^{-1} - 1 & \text{if } x = 1.
\end{cases}
\]

In \cite{FK}, Fack and Kosaki defined $\mu_s(T),$ the $generalized \ s$-$numbers$  of  an  operator $T$ in  $\mathcal L^0(\mathcal M, \tau)$ by
$$\mu_s(T)=\inf\{\|TE\|: E \mbox{ is a projection in }\mathcal M\  \mbox{with }\tau(1-E)\leq s\}, 0\leq s\leq 1.$$
The above definition admits an equivalent characterization as follows:
\[
\mu_{s}(T) =
\begin{cases}
\sup \left\{ y : \tau\left(\chi_{(y, \infty)}(|T|)\right) > s \right\} & \text{if } 0 \leq s < 1, \\
\left\| (1 + |T|)^{-1} \right\|^{-1} - 1 & \text{if } s = 1.
\end{cases}
\]

 From the definitions, there is a relation between nonincreasing rearrangements and the $s$-numbers of an element \(T \in \mathcal{L}^{0}(\mathcal{M}, \tau).\)  In fact, if \(\pi: \mathcal{L}^{\infty}[0,1] \to {W}^{*}(\mathcal{C})\) is a tracial embedding for some maximal chain \(\mathcal{C}\) of projections in \(\mathcal{M}\) with \(|T| \in {W}^{*}(\mathcal{C})^{-\tau}\), and if \(f \in \mathcal{L}^{0}(\mu)\) with \(\hat{\pi}(f) = |T|\), then for every \(s \in [0,1]\),
\[
\mu_{T}(s) = f^{*}(s).
\]
Thus there is an invertible measure-preserving transformation \(\gamma: [0,1] \to [0,1]\) such that \(f \circ \gamma = \mu_{T}\). In particular, \(f^{*}\) is independent of the choice of \(\mathcal{C}\) and \(\pi\).

\begin{definition}\label{equimeasurable}
Suppose  \(0 < S, T \in {L}^{1}(\mathcal{M}, \tau).\) Then  $S$ and $T$ are called $equimeasurable$ if \(\mu_{S} = \mu_{T}\).
\end{definition}

\section{Unitarily Invariant Norms on a finite von Neumann algebra $\mathcal M$}

\begin{definition} Suppose \(\mathcal{M}\) is a diffuse von Neumann algebra with a faithful normal tracial state \(\tau\).  A normalized unitarily invariant norm on \(\mathcal{M}\) is a norm \(\alpha: \mathcal{M} \to [0, \infty)\) such that:

(1) \(\alpha(UTV) = \alpha(T)\) for every \(T \in \mathcal{M}\) and unitaries \(U, V \in \mathcal{M}\),

(2) \(\alpha(1) = 1\).

Particularly, we say \(\alpha\) is continuous if for every \(\varepsilon > 0\), there exists a \(\delta > 0\) such that \(\alpha(P) \leq \varepsilon\) whenever \(P \in \mathcal{M}\) is a projection with \(\tau(P) < \delta\).
\end{definition}

\begin{remark}\label{relation between alpha and beta} For a \(II_1\) factor \(\mathcal{M}\), any unitarily invariant norm \(\alpha\) on $\mathcal M $ (See \cite{HN}) is obtained by first finding a symmetric gauge norm \(\beta\) on \({L}^{\infty}[0,1]\) (satisfying \(\beta(1) = 1\), \(\beta(f) = \beta(|f|) = \beta(f \circ \gamma)\) for all \(f \in {L}^{\infty}[0,1]\) and measure-preserving \(\gamma\)), then defining
\[
\alpha(T) = \beta(f) = \beta(f^{*}) \quad \text{where } \hat{\pi}(f) = |T|.
\]
Equivalently, there is a direct relation between $\alpha$ and $\beta$ via the $s$-numbers of $T$:
\[
\alpha(T) = \beta(\mu_{T}).
\]

\end{remark}

For an arbitrary von Neumann algebra \(\mathcal{M}\) with a faithful normal tracial state \(\tau\), not all norms are defined in this way. The ones that can be defined in this way are called  $ tracial\  gauge \ norms$\cite{FHNS}. Given \(\mathcal{M}\), we can find a \(II_1\) factor \(\mathcal{N}\) with trace \(\rho\) such that \(\mathcal{M} \subset \mathcal{N}\) and \(\rho|_{\mathcal{M}} = \tau\), e.g., \(\mathcal{M} \otimes \mathcal{L}_{\mathbb{F}_2}\). Then a unitarily invariant norm \(\alpha\) on \(\mathcal{M}\) is a tracial gauge norm if and only if there is a unitarily invariant norm \(\alpha_1\) on \(\mathcal{N}\) such that \(\alpha = \alpha_1|_{\mathcal{M}}\). The set of tracial gauge norms is convex and compact  with respect to the topology of pointwise convergence. For \(1 \leq p < \infty\), the norm \(\|\cdot\|_{p}\) defined by \(\|T\|_{p}=\tau(|T|^{p})^{\frac{1}{p}}\) is a tracial gauge norm.

A special family of tracial gauge norms are the {Ky-Fan} norms \({KF}_{t}\) for \(0 < t \leq 1\), defined by
\[
KF_{t}(T) = \frac{1}{t} \int_{0}^{t} \mu_{|T|}(x) \, dx.
\]

\begin{theorem}
Suppose \(N\) is a \(II_1\) factor and \(0 \leq S, T \in L^{1}(\mathcal N, \tau)\). The following are equivalent:
\begin{enumerate}
    \item \(S\) and \(T\) are equimeasurable.
    \item For any measurable function \(\varphi: [0, \infty) \to [0, \infty)\), \(\varphi(S)\) and \(\varphi(T)\) are equimeasurable.
    \item There exists a sequence \(\{U_n\}\) of unitary operators in \(\mathcal M\) such that for each \(n \in \mathbb{N}\):
    \begin{enumerate}
        \item \(S - U_n^* T U_n \in \mathcal{N}\),
        \item \(\|S - U_n^* T U_n\| \leq \frac{1}{2^{n - 1}}\).
    \end{enumerate}
    \item For every unitarily invariant norm \(\alpha\) on \(\mathcal N\), \(\alpha(S) = \alpha(T)\).
    \item For every \(0 < t \leq 1\), \(KF_{t}(S) = KF_{t}(T)\).
\end{enumerate}
\end{theorem}

\begin{proof}
\((1) \Rightarrow (2)\). Suppose \(S\) and \(T\) are equimeasurable. By Definition \ref{equimeasurable}, it follows that \(\mu_T = \mu_S\). Hence there exist tracial embeddings \(\pi, \rho: L^{\infty}[0,1] \to \mathcal M\) such that \(\hat{\pi}(\mu_T) = T\) and \(\hat{\rho}(\mu_T) = S\). Therefore, \(\hat{\pi}(\varphi \circ \mu_T) = \varphi(T)\) and \(\hat{\rho}(\varphi \circ \mu_S) = \varphi(S)\). Thus \(\mu_{\varphi(T)} = (\varphi \circ \mu_T)^* = \mu_{\varphi(S)}\), which means  \(\varphi(T)\) and \(\varphi(S)\) are equimeasurable.

\((2) \Rightarrow (3)\). Assume condition (2) holds and let \(n \in \mathbb{N}\). Define \(\varphi_n: [0, \infty) \to [0, \infty)\) by
\[
\varphi_n(t) = \frac{1}{2^n} [ 2^n t ],
\]
where \([\cdot ]\) denotes the greatest integer function. Then \(0 \leq t - \varphi_n(t) \leq \frac{1}{2^n}\) for all \(t \in [0, \infty)\), which implies $t - \varphi_n(t) \in L^\infty{[0,1]}.$ Hence \(S - \varphi_n(S)\) and \(T - \varphi_n(T) \in \mathcal M\) with
\[
\|T - \varphi_n(T)\|, \|S - \varphi_n(S)\| \leq \frac{1}{n}.
\]
Noting that \(\varphi_n(t) = \sum_{k = 0}^{\infty} \frac{k}{2^n} \chi_{[k/2^n, (k + 1)/2^n)}(t)\), we have
\[
\varphi_n(S) = \sum_{k = 0}^{\infty} \frac{k}{2^n} \chi_{[k/2^n, (k + 1)/2^n)}(S), \quad \varphi_n(T) = \sum_{k = 0}^{\infty} \frac{k}{2^n} \chi_{[k/2^n, (k + 1)/2^n)}(T).
\]
By condition (2), \(\chi_{[k/2^n, (k + 1)/2^n)}(S)\) and \(\chi_{[k/2^n, (k + 1)/2^n)}(T)\) are equimeasurable, which implies
\[
\tau\left(\chi_{[k/2^n, (k + 1)/2^n)}(S)\right) = \tau\left(\chi_{[k/2^n, (k + 1)/2^n)}(T)\right)
\]
for all \(k \geq 0\).

Since \(\mathcal{N}\) is a \(II_{1}\) factor, the projections \(\chi_{\lfloor\frac{k}{2^{k}},\,\frac{k+1}{2^{k}}\rfloor}\left(S\right)\) and \(\chi_{\lceil\frac{k}{2^{k}},\,\frac{k+1}{2^{k}}\rceil}\left(T\right)\) are Murray von Neumann equivalent. Since the projections are mutually orthogonal and add to 1, there is a unitary operator \(U_{n}\in\mathcal{M}\) such that, for every \(k\geq 0\),

\[
U_{n}\chi_{\lfloor\frac{k}{2^{k}},\,\frac{k+1}{2^{k}}\rfloor}\left(S\right)U_{n}^{*}=\chi_{\lfloor\frac{k}{2^{k}},\,\frac{k+1}{2^{k}}\rfloor}\left(T\right).
\]
Thus \(U_{n}\varphi_{n}\left(S\right)U_{n}^{*}=\varphi_{n}\left(T\right).\) Hence

\[
\|U_{n}SU_{n}^{*}-T\|\leq\|U_{n}\left(S-\varphi_{n}\left(S\right)\right)U^{*}\| +\|\varphi_{n}\left(T\right)-T\|\leq\frac{1}{2^{n-1}}.
\]

(3) \(\Rightarrow\) (4). Suppose \(\alpha\) is a unitarily invariant norm on \(\mathcal{N}\). Then \(\alpha\leq\|\cdot\|\) on \(\mathcal{N}\). Thus, for every \(n\in\mathbb{N}\),

\[
\alpha\left(S\right)=\alpha\left(U_{n}SU_{n}\right)\leq\alpha\left(U_{n}SU_{n}-T\right)+\alpha\left(T\right)\leq\|U_{n}SU_{n}-T\|+\alpha\left(T\right).
\]
Therefore \(\alpha\left(S\right)\leq\alpha\left(T\right)\). Since \(\|U_{n}SU_{n}-T\|=\|S-U_{n}^{*}TU_{n}\|\,,\) we see that \(\alpha\left(T\right)\leq\alpha\left(S\right).\)

(4) \(\Rightarrow\) (5). It follows directly from the definition of the Ky-Fan norms and Remark \ref{relation between alpha and beta}.

(5) \(\Rightarrow\) (1). For every \(x\in\left[0,t\right],\) we have

\[
\int_{0}^{t}\mu_{S}\left(x\right)dx=tKF_{t}\left(S\right)=tKF_{t}\left(T\right)=\int_{0}^{t}\mu_{T}\left(x\right)dx.
\]

Since \(S,T\in L^{1}\left(\mathcal{M},\tau\right)\), \(\mu_{S},\mu_{T}\in L^{1}\left[0,1\right],\) so differentiating the above integrals gives us that \(\mu_{S}=\mu_{T}\) a.e. (\(\mu\)). Since \(\mu_{S}\) and \(\mu_{T}\) are nondecreasing right continuous functions on \(\left[0,1\right],\) and they are equal almost everywhere, we have \(\mu_{S}=\mu_{T}\). Thus \(S\) and \(T\) are equimeasurable.

\end{proof}
\section{Connection between $\alpha$-convergence and convergence in measure }

In this section we will prove that if \(\alpha\) is a continuous unitarily invariant norm, then \(\alpha\)-convergence implies convergence in measure. This will ensure that \(L^{\alpha}\left(\mathcal{M}\right)\subset\mathcal{L}^{0}\left(\mathcal{M},\tau\right)\). If \(\alpha\) is a unitarily invariant norm on \(\mathcal{M}\), we can extend the domain of \(\alpha\) to \(\mathcal{L}^{0}\left(\mathcal{M},\tau\right)\) by

\[
\alpha\left(T\right)=\sup\left\{\alpha\left(S\right):0\leq S\leq\left|T\right|\,,S\in\mathcal{M}\right\}
\]
with \(0\leq\alpha\left(T\right)\leq\infty\). We define \(L^{\alpha}\left(\mathcal{M},\tau\right)\) to be the completion of \(\mathcal{M}\) with respect to the norm \(\alpha\), and we define \(\mathcal{L}^{\alpha}\left(\mathcal{M},\tau\right)\) to be the \(T\in\mathcal{L}^{0}\left(\mathcal{M},\tau\right)\) for which \(\alpha\left(T\right)<\infty\).

It is clear that if \(P,Q\in\mathcal{P}\left(\mathcal{M}\right),\) then
\[
\tau\left(P\lor Q\right)\leq\tau\left(P\right)+\tau\left(Q\right).
\]
Since \(\tau\left(A\right)=\|A\|_{1}\) when \(A\geq 0,\) this says \(\|P\lor Q\|_{1}\leq\|P\|_{1}+\|Q\|_{1}\). We prove that this holds for any unitarily equivalent norm \(\alpha\) on \(\mathcal{M}\).

\begin{lemma}\label{monotone of norm}
Suppose \(\mathcal{M}\) is a von Neumann algebra and \(P,Q\) are projections in \(\mathcal{M}\). Then there is a unitary operator \(U\in\mathcal{M}\) such that

\begin{enumerate}
    \item \(P\) and \(U^{*}QU\) commute;
    \item \(P\lor Q=P\vee(U^{*}QU)\,;\)
    \item If \(\alpha\) is any unitarily invariant norm on \(\mathcal{M}\), then
    \[
    \alpha\left(P\lor Q\right)\leq\alpha\left(P\right)+\alpha\left(Q\right).
    \]
\end{enumerate}
\end{lemma}

\begin{proof}
From Halmos's representation of two projections, we can write
\[
P = I \oplus 0 \oplus I \oplus 0 \oplus \begin{pmatrix} I & 0 \\ 0 & 0 \end{pmatrix}
\]
and
\[
Q = I \oplus 0 \oplus 0 \oplus I \oplus \begin{pmatrix} c^{2} & cs \\ cs & s^{2} \end{pmatrix},
\]
where \(0 \leq s, c \leq 1\), \(\ker s = \ker(1 - s) = \ker c = \ker(1 - c) = 0\), \(sc = cs\), and \(s^{2} + c^{2} = 1\). In this case,
\[
W^{*}(P, Q) = \mathbb{C}I \oplus \mathbb{C}I \oplus \mathbb{C}I \oplus \mathbb{C}I \oplus \mathbb{M}_{2}(W^{*}(s, c)).
\]
There is a unitary
\[
U = I \oplus I \oplus I \oplus I \oplus \begin{pmatrix} s & -c \\ c & s \end{pmatrix} \in W^{*}(P, Q)
\]
such that
\[
Q^{\prime} = U^{*}QU = I \oplus 0 \oplus 0 \oplus I \oplus \begin{pmatrix} 0 & 0 \\ 0 & I \end{pmatrix}.
\]
Then \(PQ^{\prime} = Q^{\prime}P\) and
\[
P \lor Q^{\prime} = P \lor Q = I \oplus 0 \oplus I \oplus I \oplus \begin{pmatrix} I & 0 \\ 0 & I \end{pmatrix}.
\]

If \(\alpha\) is a unitarily invariant norm on \(\mathcal{M}\), then \(P \lor Q^{\prime} \leq P + Q^{\prime}\) since \(P\) and \(Q^{\prime}\) commute. Therefore
\[
\alpha(Q^{\prime}) = \alpha(U^{*}QU) = \alpha(Q),
\]
and
\[
\alpha(P \lor Q) = \alpha(P \lor Q^{\prime}) \leq \alpha(P) + \alpha(Q^{\prime}) = \alpha(P) + \alpha(Q).
\]
\end{proof}
We now prove that \(\alpha\)-convergence in \(\mathcal{L}^{0}(\mathcal{M}, \tau)\) implies \(\tau\)-convergence.

\begin{theorem}\label{alpha convergence implies in measure}
Suppose \(\alpha\) is a continuous unitarily invariant norm on a von Neumann algebra \(\mathcal{M}\) with a faithful normal tracial state \(\tau\). If \(\{T_{n}\}\) is a sequence in \(\mathcal{M}\) and \(\alpha(T_{n}) \to 0\), then \(T_{n} \stackrel{\tau}{\rightarrow} 0\).
\end{theorem}

\begin{proof}
Suppose \(\alpha(T_{n}) \to 0\). Let \(0 < \varepsilon < 1\), and let \(P_{n} = \chi_{[\varepsilon, \infty)}(|T_{n}|)\). For each \(n \in \mathbb{N}\), let
\[
f_{n}(t) = \chi_{[\varepsilon, \infty)}(t) \frac{t}{1 + t}
\]
on \([0, 1]\). We have
\[
\left(1 - \frac{1}{\varepsilon}\right) \chi_{[\varepsilon, \infty)}(t) \leq f_{n}(t) \leq t.
\]
Thus,
\[
|T_{n}| \geq P_{n}|T_{n}|(1 + |T_{n}|)^{-1} \geq \left(1 - \frac{1}{\varepsilon}\right) P_{n}.
\]
Hence,
\[
\alpha(T_{n}) \geq \left(1 - \frac{1}{\varepsilon}\right) \alpha(P_{n}),
\]
which implies \(\alpha(P_{n}) \to 0\).\\

\noindent\textbf{Claim:} \(\tau(P_{n}) \to 0\).

Assume, via contradiction, that \(\tau(P_{n}) \not\to 0\). Then, by considering an appropriate subsequence, we can assume that there is a \(\delta > 0\) such that, for every \(n \geq 1\),
\[
\alpha(P_{n}) < \frac{1}{2^{n}} \quad \text{and} \quad \tau(P_{n}) \geq \delta.
\]
Suppose \(N \in \mathbb{N}\). It follows from Lemma \ref{monotone of norm} and induction that
\[
\alpha\left(\bigvee_{N \leq k \leq n} P_{k}\right) \leq \sum_{N \leq k \leq n} \alpha(P_{k})
\]
for each integer \(n\).  Noting that
\[
\bigvee_{N \leq k \leq n} P_{k} \rightarrow \bigvee_{N \leq k < \infty} P_{k} = Q_{N}
\]
in the weak* topology, so
\[
\tau(Q_{N} - \bigvee_{N \leq k \leq n} P_{k}) \to 0.
\]
Since \(\alpha\) is continuous,
\[
|\alpha(Q_{N}) - \alpha(\bigvee_{N \leq k \leq n} P_{k})| \leq \alpha(Q_{N} - \bigvee_{N \leq k \leq n} P_{k}) \to 0.
\]
Thus,
\[
\alpha(Q_{N}) = \lim_{n \to \infty} \alpha\left(\bigvee_{N \leq k \leq n} P_{k}\right) \leq \lim_{n \to \infty} \sum_{N \leq k \leq n} \alpha(P_{k}) \leq \sum_{k=N}^{\infty} \frac{1}{2^{k}} = \frac{1}{2^{N}}.
\]

Also, observing that \(Q_{1} \geq Q_{2} \geq Q_{3} \geq \cdots\), so there is a projection \(Q\) such that \(Q_{n} \to Q\) in the weak* topology. Thus, \(\tau(Q - Q_{N}) \to 0\). Since \(\tau\) is continuous, we know \(\alpha(Q - Q_{N}) \to 0\). Therefore,

\[
\alpha(Q) = \lim_{N \to \infty} \alpha(Q_{N}) = 0.
\]

However,

\[
\tau(Q) = \lim_{N \to \infty} \tau(Q_{N}) \geq \limsup_{N \to \infty} \tau(P_{N}) \geq \delta.
\]
Thus \(Q \neq 0\), but \(\alpha(Q) = 0\), which contradicts the fact that \(\alpha\) is a norm. This proves the claim, i.e., $\tau(P_n)=\tau(\chi_{[\varepsilon, \infty)}(|T_{n}|))\rightarrow 0$. By Lemma \ref{Convergence in measure Lemma}, we see \(T_{n} \stackrel{\tau}{\rightarrow} 0\).
\end{proof}
\begin{corollary}
Suppose \(\alpha\) is a continuous unitarily invariant norm on a von Neumann algebra \(\mathcal{M}\) with a faithful normal tracial state \(\tau\). Then we have
\[
L^{\alpha}(\mathcal{M}, \tau) \subset \mathcal{L}^{\alpha}(\mathcal{M}, \tau) \subset \mathcal{L}^{0}(\mathcal{M}, \tau).
\]
\end{corollary}

\bigskip

\section{Dominated Convergence Theorems on $L^\alpha(\mathcal M, \tau)$}\label{section:counterexamples}

In this section, we prove a very general noncommutative dominated convergence theorem, with a key idea being a very relaxed notion of domination. Suppose \(\mathcal{M}\) is a diffuse von Neumann algebra with a faithful normal trace \(\tau\). Let \(\mathcal{P}(\mathcal{M})\) denote the set of all projections in \(\mathcal{M}\).

\begin{definition}\label{equicontinuous}
Suppose \(\alpha\) is a continuous normalized unitarily invariant norm on \(\mathcal{M}\). Then \(\mathcal{S} \subset L^{\alpha}(\mathcal{M}, \tau)\) is called \(\alpha\)-$equicontinuous,$ if for every \(\varepsilon > 0\), there is a \(\delta > 0\) such that, for every \(S \in \mathcal{S}\) and every projection \(P \in \mathcal{M}\),

\[
\tau(P) < \delta \Longrightarrow \alpha(PS) \leq \varepsilon.
\]
\end{definition}

We define \(d_{\mathcal{S}}: (0, 1] \rightarrow (0, 1]\) by
\[
d_{\mathcal{S}}(\varepsilon) = \sup \left\{ 0 < \delta \leq 1 : \forall S \in \mathcal{S} \text{ and } \forall P \in \mathcal{P}(\mathcal{M}), \ [\tau(P) < \delta \Longrightarrow \alpha(PS) \leq \varepsilon] \right\}.
\]
That is, \(d_{\mathcal{S}}(\varepsilon)\) is the largest \(\delta\) with \(0 < \delta \leq 1\) such that \(\tau(P) < \delta \Longrightarrow \alpha(PS) \leq \varepsilon\) is true for all \(S \in \mathcal{S}\) and all \(P \in \mathcal{P}(\mathcal{M})\).

\begin{definition}\label{alpha dominating set} The function \(d_{\mathcal{S}},\) as defined above, is called the \(\alpha\)-dominating function for \(\mathcal{S}\). We say an element \(T \in \mathcal{L}^{0}(\mathcal{M}, \tau)\) is \(\alpha\)-dominated by \(\mathcal{S}\) if and only if
\[
\forall P \in \mathcal{P}(\mathcal{M}), \quad [\tau(P) < d_{\mathcal{S}}(\varepsilon) \Longrightarrow \alpha(PT) \leq \varepsilon],
\]
and we denote the collection of all such operators by
 $$\mathcal{D}_{\alpha}(\mathcal{S}) = \mathcal{D}_{\alpha}(\mathcal{S}, \mathcal{M}) = \left\{ T \in \mathcal{L}^{0}(\mathcal{M}, \tau) : T \text{ is \(\alpha\)-dominated by \(\mathcal{S}\)} \right\}.$$
\end{definition}

\begin{lemma}\label{alpha equicontinuous}
Suppose \(\alpha\) is a continuous normalized unitarily invariant norm and \(T \in L^{\alpha}(\mathcal{M})\). Then \(\{T\}\) is \(\alpha\)-equicontinuous.
\end{lemma}

\begin{proof}
Suppose \(T \in L^{\alpha}(\mathcal{M})\).  Then for any \(\varepsilon > 0\), there is an \(0 \neq A \in \mathcal{M}\) such that \(\alpha(T - A) < \varepsilon/2\). Since \(\alpha\) is continuous, there is a \(\delta > 0\) such that if \(P \in \mathcal{P}(\mathcal{M})\) and \(\tau(P) < \delta\), then \(\alpha(P) < \frac{\varepsilon}{2\|A\|}\), which implies
\[
\alpha(PT) \leq \alpha(P(T - A)) + \alpha(PA) \leq \|P\|\alpha(T - A) + \alpha(P)\|A\| \leq \varepsilon/2 + \varepsilon/2 = \varepsilon.
\] This show that $\{T\}$ is \(\alpha\)-equicontinuous.
 \end{proof}

\begin{theorem}\label{DCT}
Suppose \(\mathcal{M}\) is a diffuse von Neumann algebra with a faithful normal trace \(\tau\), \(\alpha\) is a continuous normalized unitarily invariant norm on \(\mathcal{M}\), and \(\mathcal{S} \subset L^{\alpha}(\mathcal{M}, \tau)\) is \(\alpha\)-equicontinuous. Then

\begin{enumerate}
    \item \(\mathcal{S} \subset \mathcal{D}_{\alpha}(\mathcal{S})\).
    \item \(\mathcal{D}_{\alpha}(\mathcal{S})\) is convex.

    \item[(3)] If \(T \in \mathcal{D}_{\alpha}(\mathcal{S})\) and \(A, B \in \mathcal{M}\) with \(\|A\|, \|B\| \leq 1\), then \(ATB \in \mathcal{D}_{\alpha}(\mathcal{S})\).

    \item[(4)] If \(S, S_{1}, S_{2}, \ldots \in \mathcal{D}_{\alpha}(\mathcal{S})\), then \(S_{n} \to S\) in measure if and only if
    \[
    \alpha(S_{n} - S) \to 0.
    \]
    \item[(5)] The \(\alpha\)-topology and the topology of convergence in measure coincide on \(\mathcal{D}_{\alpha}(\mathcal{S})\).

    \item[(6)] \(\mathcal{D}_{\alpha}(\mathcal{S}) \subset L^{\alpha}(\mathcal{M},\tau)\).

    \item[(7)] \(\mathcal{D}_{\alpha}(\mathcal{S})\) is closed in convergence in measure and in the \(\alpha\)-topology.

    \item[(8)] If \(\gamma\) is a \(\ast\)-automorphism of \(\mathcal{M}\) such that \(\tau \circ \gamma = \tau\) and \(\alpha \circ \gamma = \alpha\), then \(\gamma(T) \in \mathcal{D}_{\alpha}(\mathcal{S})\) whenever \(T \in \mathcal{D}_{\alpha}(\mathcal{S})\).
\end{enumerate}
\end{theorem}

\begin{proof} (1). This is obvious.

(2). Suppose \(A, B \in \mathcal{D}_{\alpha}(\mathcal{S})\), \(\varepsilon > 0\), and \(P \in \mathcal{P}(\mathcal{M})\) with \(\tau(P) < d_{\mathcal{S}}(\varepsilon)\), suppose \(0 \leq t \leq 1\). Then
\[
\alpha(P((1 - t)A + tB)) \leq (1 - t)\alpha(PA) + t\alpha(PB) \leq (1 - t)\varepsilon + t\varepsilon = \varepsilon.
\] Therefore \(\mathcal{D}_{\alpha}(\mathcal{S})\) is convex.

(3). It is enough to prove the result when \(\|A\| < 1\) and \(\|B\| < 1\). Choose \(n_{0} \in \mathbb{N}\) with \(n_{0} > 2\) such that \(\|A\| < \frac{n - 2}{n}\). Then, for \(n \geq n_{0}\), \(\left\|\frac{n}{n - 2}A\right\| < \frac{n}{n - 2}\). Thus, by the Russo-Dye theorem [\cite{DR}, Theorem I.8.4], there are unitary operators \(U_{1}, \ldots, U_{n} \in \mathcal{M}\) such that
\[
\frac{n - 2}{n}A = \frac{1}{n}[U_{1} + \cdots + U_{n}].
\]
Suppose \(\varepsilon > 0\) and \(P \in \mathcal{P}(\mathcal{M})\) with \(\tau(P) < d_{\mathcal{S}}(\varepsilon)\). Then, since each \(U_{k}^{*}PU_{k} \in \mathcal{P}(\mathcal{M})\) and \(\tau(U_{k}^{*}PU_{k}) = \tau(P) < d_{\mathcal{S}}(\varepsilon)\), we see
\begin{align*}
\frac{n - 2}{n}\alpha(PAT)& = \alpha(P(A)T) = \alpha\left(P\frac{1}{n}[U_{1} + \cdots + U_{n}]T\right)
\\&\leq \frac{1}{n}\sum_{k=1}^{n}\alpha(PU_{k}T) = \frac{1}{n}\sum_{k=1}^{n}\alpha(U_{k}U_{k}^{*}PU_{k}T) \\&= \frac{1}{n}\sum_{k=1}^{n}\alpha(U_{k}^{*}PU_{k}T)
\leq \frac{1}{n}\sum_{k=1}^{n}\varepsilon = \varepsilon.
\end{align*}
Since \(n \in \mathbb{N}\) was arbitrary, we see that$\alpha(PAT) \leq \varepsilon.$
Hence, \(AT \in \mathcal{D}_{\alpha}(\mathcal{S})\). In a similar fashion, we can show that \(TB \in \mathcal{D}_{\alpha}(\mathcal{S})\).

(4). It follows from Theorem \ref{alpha convergence implies in measure} that \(\alpha\)-convergence implies \(\tau\)-convergence. Now suppose \(S_{1}, S_{2}, \ldots \in \mathcal{D}_{\alpha}(\mathcal{S})\) and \(S_{n} \to 0\) in measure. By Lemma \ref{Convergence in measure Lemma}, there is a sequence \(\{P_{n}\}\) of projections in \(\mathcal{M}\) such that \(\tau(P_{n}) \to 0\) and \(\|(1 - P_{n})S_{n}\| \to 0\). If \(\varepsilon > 0\), then there is an \(N \in \mathbb{N}\) such that when \(n \geq N\), \(\tau(P_{n}) < \delta_{\mathcal{S}}(\varepsilon/2)\) and \(\|(1 - P_{n})S_{n}\| < \varepsilon/2\), which implies
\[
\alpha(S_{n}) \leq \alpha(P_{n}S_{n}) + \alpha((1 - P_{n})S_{n}) \leq \varepsilon/2 + \|(1 - P_{n})S_{n}\| < \varepsilon.
\]
Hence, \(\alpha(S_{n}) \to 0\). Next we suppose \(\{T_{n}\}\) is a sequence in \(\mathcal{D}_{\alpha}(\mathcal{S})\) and \(T \in \mathcal{L}^{0}(\mathcal{M}, \tau)\) with \(T_{n} \xrightarrow{\tau} T\). Noting that $\frac{1}{2}(T_{n} - T_{m})=\frac{1}{2}(T_{n} - T)-\frac{1}{2}(T_{m} - T),$ since \(T_{n} \xrightarrow{\tau} T\) , it follows that $\frac{1}{2}(T_{n} - T_{m})\xrightarrow{\tau} 0.$ By the discussion above, it shows that \(\alpha(T_{n} - T_{m}) \to 0\). Hence,  $\{T_n\}$ is a Cauchy sequence in $L^\alpha(\mathcal M, \tau).$ Then there is an \(A \in L^{\alpha}(\mathcal{M},\tau)\) such that \(\alpha(T_{n} - A) \to 0\), thus, \(T_{n} \xrightarrow{\tau} A\), which means \(A = T\). Therefore $\alpha(T_n-T)\rightarrow 0.$ This completes the proof.

(5). This follows immediately from (4).

(6). Suppose \(T \in \mathcal{D}_{\alpha}(\mathcal{S})\). It follows from part (3) that \(|T| \in \mathcal{D}_{\alpha}(\mathcal{S})\) and that, for every positive integer \(n\), \(\chi_{\{0, n\}}(|T|)|T| \in \mathcal{M} \cap \mathcal{D}_{\alpha}(\mathcal{S})\) and \(\chi_{\{0, n\}}(|T|)|T| \xrightarrow{\tau} |T|\). Thus, by part (5),
\[
\alpha\left(\chi_{[0,n)}\left(|T|\right)|T| - |T|\right) \to 0.
\]
Thus \(|T| \in L^{\alpha}(\mathcal{M}, {\tau})\), which implies \(T \in L^{\alpha}(\mathcal{M}, \tau)\).

(7). We see from part (5) and its proof that \(\mathcal{D}_{\alpha}(\mathcal{S})\) is \(\alpha\)-closed if and only if it is \(\tau\)-closed. Suppose \(\{T_{n}\}\) is a sequence in \(\mathcal{D}_{\alpha}(\mathcal{S})\), \(T \in L^{\alpha}(\mathcal{M}, \tau)\), and \(\alpha(T_{n} - T) \to 0\). Suppose \(\varepsilon > 0\) and \(P\) is a projection in \(\mathcal{M}\) with \(\tau(P) < d_{\mathcal{S}}(\varepsilon)\). Since
\[
0 \leq \alpha(PT_{n} - PT) \leq \|P\|\alpha(T_{n} - T) \to 0,
\]
we see that
\[
\alpha(PT) = \lim_{n \to \infty} \alpha(P_{n}T) \leq \varepsilon.
\]
Thus \(T \in \mathcal{D}_{\alpha}(\mathcal{S})\).

(8). Suppose \(T \in \mathcal{D}_{\alpha}(\mathcal{S})\), \(\varepsilon > 0\), and \(P \in \mathcal{M}\) is a projection such that \(\tau(P) < d_{\mathcal{S}}(\varepsilon)\). Then \(\tau(\gamma^{-1}(P)) = \tau(P) < \varepsilon\), and therefore,
\[
\alpha(P\gamma(T)) = \alpha(\gamma^{-1}(P\gamma(T))) = \alpha(\gamma^{-1}(P)T) \leq \varepsilon.
\]
\end{proof}

The following corollary states that the dominated convergence theorem holds for $L^\alpha(\mathcal M, \tau),$ but fails in the larger space $\mathcal L^\alpha(\mathcal M, \tau).$ This fact shows that two spaces are different, which is a key distinction from the classical $L^p$-cases since $L^p(\mathcal M, \tau)=\mathcal L^p(\mathcal M, \tau)$ for $1\leq p\leq \infty$.

\begin{corollary}
Suppose \(0 \leq T \in \mathcal L^{0}(\mathcal{M}, \tau)\) and \(\alpha(T) < \infty\), but \(T \notin L^{\alpha}(\mathcal{M},\tau)\). Then there is a sequence \(\{P_{n}\}\) in \(\mathcal{P}(\mathcal{M})\) such that \(\tau(P_{n}) \to 0\) and \(\alpha(P_{n}T) \nrightarrow 0\).
\end{corollary}

\begin{proof}
Let \(Q_{n} = \chi_{[0,n]}(T)\) and let \(P_{n} = 1 - Q_{n}\). Since \(\chi_{[0,n]} \to 1\) in measure in \(L^{\infty}[0,1]\), we have from Theorem \ref{alpha convergence implies in measure} that \(Q_{n} \xrightarrow{\tau} 1\) and thus \(P_{n} \xrightarrow{\tau} 0\) . If \(\alpha(P_{n}T) \to 0\), then we would have
\[
\alpha(T - Q_{n}T) \to 0.
\]
However, \(Q_{n}T = \chi_{[0,n]}(T)T \in \mathcal{M}\), which would imply \(T \in L^{\alpha}(\mathcal{M},\tau)\). Thus \(\alpha(P_{n}T) \nrightarrow 0\).
\end{proof}

The following result was proved in \cite{CHS} for continuous unitarily invariant norms \(\alpha\) such that \(\alpha \geq \|\cdot\|_{1}\).

\begin{corollary}
If \(\alpha\) is a continuous unitarily invariant norm on \(\mathcal{M}\), and \(\mathcal{B}\) is the closed unit ball of \(\mathcal{M}\), then the \(\alpha\)-topology and the topology of convergence in measure coincide on \(\mathcal{B}\). Also, if \(\mathcal{S} \subset \mathcal{M}\) is an \(\alpha\)-closed linear subspace of \(L^{\alpha}(\mathcal{M}, \tau)\), then \(S \cap \mathcal{M}\) is weak*-closed in \(\mathcal{M}\).
\end{corollary}

\begin{proof}
The first statement follows from the fact that \(\mathcal{B} \subset \mathcal{D}_{\alpha}(\{1\})\) and from parts (3) and (5) of Theorem \ref{DCT}. The second statement in the corollary follows from the Krein-Shmulyan theorem and the fact that, since \(\mathcal{B} \cap \mathcal{M}\) is convex, it is \(\|\cdot\|_{2}\)-closed if and only if it is weakly \(\|\cdot\|_{2}\)-closed if and only if it is weak*-closed.
\end{proof}

The following example further illustrates that the dominating condition in Theorem \ref{DCT} is substantially weaker than the classical DCT.

\begin{example}
Suppose \(\mathcal{M}\) is a \(II_{1}\) factor and \(n \in \mathbb{N}\) and \(P_{1} \in \mathcal{M}\) is a projection with \(\tau(P_{1}) = \frac{1}{n}\). Then \(\mathcal{D}(\{n^{2}P_{1}\})\) contains all operators \(T \in \mathcal{M}\) with \(\|T\| \leq n\). To see this, we find projections \(P_{2}, \ldots, P_{n} \in \mathcal{M}\) such that \(\tau(P_{k}) = \frac{1}{n}\) for each \(k\) and \(\sum_{k=1}^{n} P_{k} = 1\). Since \(\mathcal{M}\) is a \(II_{1}\) factor, each \(n^{2}P_{k}\) is unitarily equivalent to \(n^{2}P_{1}\), so
\[
n = \frac{1}{n} \sum_{k=1}^{n} n^{2} P_{k} \in \mathcal{D}_{\alpha}(\{n^{2}P_{1}\}).
\]
If \(T \in \mathcal{M}\) and \(\|T\| = n\), then \(T = \left(\frac{T}{n}\right)n \in \mathcal{D}_{\alpha}(\{n^{2}P_{1}\})\). Whereas if \(|T| \leq n^{2}P_{1}\), we have \(|T| = P_{1}|T|P_{1}\). Thus, our notion of domination is much less restrictive than the notion in the classical version of the dominated convergence theorem.
\end{example}

The automorphisms on \(\mathcal{M}\) for which statement (8) in Theorem \ref{DCT} holds for every unitarily invariant norm \(\alpha\) are precisely those that fix the center \(\mathcal{Z}(\mathcal{M})\). Note that the von Neumann algebra \(\mathcal{M}\) in the following Theorem does not need to act on a separable Hilbert space.

\begin{theorem}
Suppose \(\mathcal{M}\) is a finite von Neumann algebra and \(\gamma\) is a \(*\)-automorphism of \(\mathcal{M}\). The following are equivalent:

\begin{enumerate}
    \item[(1)] \(\gamma(S) = S\) for every \(S \in \mathcal{Z}(\mathcal{M})\), i.e., \(\gamma\) fixes \(\mathcal{Z}(\mathcal{M})\).

    \item[(2)] \(\Phi \circ \gamma = \Phi\), where \(\Phi: \mathcal{M} \rightarrow \mathcal{Z}(\mathcal{M})\) is the center-valued trace on \(\mathcal{M}\).

    \item[(3)] For every \(A = A^{*} \in \mathcal{M}\), there is a sequence \(\{W_{n}\}\) of unitary operators in \(\mathcal{M}\) such that \(\|W_{n}^{*}AW_{n} - \gamma(A)\| \to 0\).

    \item[(4)] \(\alpha(\gamma(T)) = \alpha(T)\) for every unitarily invariant norm \(\alpha\) on \(\mathcal{M}\).

    \item[(5)] \(\tau \circ \gamma = \tau\) for every normal tracial state \(\tau\) on \(\mathcal{M}\).
\end{enumerate}
\end{theorem}

\begin{proof}
\((1) \Longrightarrow (2)\). Suppose \(\gamma \in \text{Aut}(\mathcal{M})\) and \(\gamma(S) = S\) for every \(S \in \mathcal{Z}(\mathcal{M})\). Suppose \(A = A^{*} \in \mathcal{M}\). Suppose \(T \in C^{*}(A)\). Suppose \(\varepsilon > 0\). It follows from Dixmier's averaging theorem \cite{Dixmier} that we can choose unitary elements \(U_{1}, \ldots, U_{n} \in \mathcal{M}\) such that
\[
\left\|\Phi(A) - \frac{1}{n}\sum_{k=1}^{n} U_{k}^{*}AU_{k}\right\| < \varepsilon.
\]
Since \(\gamma(\Phi(A)) = \Phi(A)\), we have
\[
\left\|\Phi(A) - \frac{1}{n}\sum_{k=1}^{n} \gamma(U_{k}^{*})\gamma(A)\gamma(U_{k})\right\| < \varepsilon.
\]
Thus, applying \(\Phi\), we obtain
\[
\|\Phi(A) - \Phi(\gamma(A))\| = \left\|\Phi\left(\Phi(A) - \frac{1}{n}\sum_{k=1}^{n} \gamma(U_{k}^{*})\gamma(A)\gamma(U_{k})\right)\right\| < \varepsilon.
\]
Since \(\varepsilon > 0\) was arbitrary, we have \(\Phi(\gamma(A)) = \Phi(A)\).

\((2) \Longrightarrow (3)\). Suppose \(\Phi \circ \gamma = \Phi\). Then \(\text{id}|_{C^{*}(A)}, \gamma|_{C^{*}(A)}: C^{*}(A) \rightarrow \mathcal{M}\) are unital \(*\)-homomorphisms and \(\Phi \circ \text{id}|_{C^{*}(A)} = \Phi \circ \gamma|_{C^{*}(A)}\). It follows from \cite{LHL} that there is a sequence \(\{W_{n}\}\) of unitary operators in \(\mathcal{M}\) such that \(\|W_{n}^{*}AW_{n} - \gamma(A)\| \to 0\).

\((3) \Longrightarrow (4)\). Suppose (3) holds and \(\alpha\) is any unitarily invariant norm on \(\mathcal{M}\). If \(T \in \mathcal{M}\), we can choose a sequence \(\{U_{n}\}\) of unitary operators in \(\mathcal{M}\) such that \(\|U_{n}^{*}|T|U_{n} - \gamma(|T|)\| \to 0\). Then
\begin{align*}
|\alpha(T) - \alpha(\gamma(T))| &= |\alpha(|T|) - \alpha(\gamma(|T|))| = \alpha(|T| - \gamma(|T|)) \\&\leq \||T| - \gamma(|T|)\| = \lim_{n \to \infty} \|U_{n}^{*}|T|U_{n} - \gamma(|T|)\| = 0,
\end{align*}
which means \(\alpha(\gamma(T)) = \alpha(T)\).

\((4) \Longrightarrow (5)\). Suppose (4) holds and \(\tau\) is a normal tracial state on \(\mathcal{M}\). Then the norm \(\alpha(S) = \frac{1}{2}\tau(|S|) + \frac{1}{2}\|S\|\) defines a unitarily invariant norm on \(\mathcal{M}\). Thus, for any \(0 \leq S \in \mathcal{M}\),
\[
\frac{1}{2}\tau(\gamma(S)) + \frac{1}{2}\|\gamma(S)\| = \frac{1}{2}\tau(S) + \frac{1}{2}\|S\|.
\]
Since \(\|\gamma(S)\| = \|S\|\), we have \(\tau(S) = (\tau \circ \gamma)(S)\) for every \(S \geq 0\). Hence \(\tau = \tau \circ \gamma\).

\((5) \Longrightarrow (1)\) Suppose $\tau\circ \gamma=\tau$ for every normal tracial state $\tau$ on $\mathcal M.$ Take any central element $S\in \mathcal Z(M) ,$ set $D=\gamma(S)-S.$ Since $\gamma $ is an automorphism, it follows that  $D\in \mathcal Z(M).$ Thus $\tau(D)=\tau(\gamma(S)-S)=\tau(\gamma S)-\tau(S)=0.$ Noting that normal tracial state $\tau$ separates points in the center of a finite von Neumann algebra, which implies $D=0.$ Hence $\gamma(S)=S$ for all $S\in \mathcal Z(M).$ This completes the proof.
\end{proof}

 Our contribution of this result is to add a new norm-theoretic criterion to the earlier works of Kadison and Haagerup on the structure of automorphisms:
the class of center-preserving $*$-automorphisms (equivalently, $*$-automorphisms leaves trace invariant) coincides with those $*$-automorphisms preserving  every unitarily invariant norm on $\mathcal M$, which connects Dixmier's averaging theorem, spectral measure theory and norm invariance on von Neumann algebras.

\end{document}